\documentclass[11pt]{article}
\usepackage{amsmath,amssymb,amsthm,amsfonts,hhline,color}
\usepackage{etoolbox, textcomp}
\patchcmd{\thebibliography}{\section*{\refname}}{}{}{}
\usepackage{amscd}
\usepackage{hyperref}
\usepackage{tikz} 
\usepackage[matrix,arrow,curve,frame]{xy}

\usetikzlibrary{arrows,shapes}
\def\cA{\mathcal{A}}
\def\cC{{\mathcal C}}
\def\cCA{{\mathcal C_A}}
\def\cCAloc{{\mathcal C_A^{\text{loc}}}}

\def\Z{{\mathbb Z}}

\def\Q{{\mathbb Q}}

\def\g{{\mathfrak g}}
\def\WL{{\mathbf L}}
\def\AL{{\mathbb L}}

\title{Fusion categories for affine vertex algebras\\ at admissible levels }

\author{Thomas Creutzig\thanks{Department of Mathematical and Statistical Sciences, University of Alberta,
Edmonton, Alberta  T6G 2G1, Canada and\newline Research Institute for Mathematical Sciences, Kyoto University, Kyoto Japan 606-8502.
\newline email: creutzig@ualberta.ca}}

\date{}

\begin{document}

\bibliographystyle{amsalpha}

\theoremstyle{plain}
\newtheorem{introthm}{Main Theorem}
\newtheorem{obs}{Observation}
\newtheorem{thm}{Theorem}[section]
\newtheorem{prop}[thm]{Proposition}
\newtheorem{lem}[thm]{Lemma}
\newtheorem{cor}[thm]{Corollary}
\newtheorem{introcor}{Corollary}
\newtheorem{conj}[thm]{Conjecture}

\theoremstyle{definition}
\newtheorem{defi}[thm]{Definition}
\newtheorem{rem}[thm]{Remark}
\newtheorem{ex}[thm]{Example}

\newcommand{\one}{\mathbf 1}
\newcommand{\Id}{\text{Id}}
\newcommand {\CC}{\mathbb{C}}
\newcommand {\ZZ}{\mathbb{Z}}
\newcommand{\cQ}{\mathcal{Q}}
\newcommand{\cH}{\mathcal{H}}
\newcommand {\tr}{\text{tr}}
\newcommand {\ch}{\text{ch}}
\newcommand {\sch}{\text{sch}}
\newcommand {\sltwo}{\mathfrak{sl}_2}
\newcommand {\cW}{\mathcal{W}}
\newcommand {\cG}{\mathcal{G}}
\newcommand {\cM}{\mathcal{M}}
\newcommand {\cVec}{\text{Vec}}
\newcommand {\cX}{\mathcal{X}}
\newcommand {\cB}{\mathcal{B}}
\newcommand {\cD}{\mathcal{D}}
\newcommand {\cS}{\mathcal{S}}
\newcommand {\cO}{\mathcal{O}}
\newcommand {\cF}{\mathcal{F}}
\newcommand{\cL}{\mathcal{L}}
\newcommand{\cbarH}{\overline{\mathcal{H}}}
\newcommand{\func}[1]{\mathcal F_{#1}}
\newcommand{\HHom}[2]{\text{Hom}_{#1}\left( #2 \right)}
\newcommand {\cCop}{\mathcal{C}^{\text{op}}}
\newcommand {\cCrev}{\mathcal{C}^{\text{rev}}}
\newcommand {\voa}{vertex operator algebra}
\newcommand {\voas}{vertex operator algebras}
\newcommand {\vosa}{vertex operator superalgebra}
\newcommand {\vosas}{vertex operator superalgebras}
\newcommand{\Sing}{M(p)}
\newcommand{\Trip}{W(p)}
\newcommand{\Hom}{\text{Hom}}
\newcommand{\intHom}{\underline{\text{Hom}}}
\newcommand{\End}{\text{End}}
\newcommand{\FP}{\text{FP}}
\newcommand{\obj}{\text{Obj}}
\newcommand{\svec}{\text{sVec}}

\tikzset{ar/.style={<-}}%,thick,shorten <=8pt,shorten >=8pt,>=stealth}}
\tikzset{br/.style={->}}%,thick,shorten <=8pt,shorten >=8pt,>=stealth}} 

\newcommand{\hopflink}{{\text{\textmarried}}}

\renewcommand{\baselinestretch}{1.2}

\maketitle

\begin{abstract}
The main result is that the category of ordinary modules of an affine \voa{} of a simply laced Lie algebra at admissible level is rigid and thus a braided fusion category. If the level satisfies a certain coprime property then it is even a modular tensor category. 
In all cases open Hopf links coincide with the corresponding   normalized S-matrix entries of torus one-point functions. 
This is interpreted as a Verlinde formula beyond rational \voas.

 A preparatory Theorem is a convenient formula for the fusion rules of rational principal W-algebras of any type.
\end{abstract}

\newpage

%\tableofcontents

\newpage

\section{Introduction}

Vertex algebras are a rigorous formulation of chiral algebras of two dimensional conformal field theories of physics. They appear in many interesting problems of both mathematics and physics. In our context the relation to braided tensor categories and modular forms is of interest. By now one understands this well in the instance of strongly rational \voas{} \cite{H1, H5}. In this case one has a modular tensor category and Verlinde's formula holds \cite{H5} (conjectured by Verlinde \cite{V}), i.e. normalized Hopf links coincide with the corresponding  normalized S-matrix entries of torus one-point functions. 
The category of a \emph{nice} class of modules of  a \voa{} is in general expected to form a rigid vertex tensor category.
But already proving the existence of a vertex tensor category is rather difficult as one has to verify that quite a few assumptions hold. This theory has been developed by Huang, Lepowsky and Zhang in a series of many papers \cite{HLZ1}-\cite{HLZ8} and in joint work with Huang and Yang all these assumptions have been verified to hold for the category $\cO_\ell(\g)$ of ordinary modules of affine \voas{} $\AL_\ell(\g)$ at admissible level $\ell$ \cite{CHY}. Tomoyuki Arakawa has proven that this category is semi-simple \cite{Ar3} and so it is natural to wonder if this category is even a ribbon or modular tensor category.
In the case of $\g=\sltwo$ we proved that $\cO_\ell(\sltwo)$ is always ribbon and modular if and only if the denominator of $\ell$ is odd \cite{CHY}. The  natural conjecture is then a similar outcome for any $\cO_\ell(\g)$ and indeed:
\begin{introthm}
Let $\g$ be simply-laced and let $\ell$ be an admissible level for $\g$, then the category $\cO_\ell(\g)$ is ribbon.
\end{introthm}
Note that by a ribbon category we mean a rigid, braided, semi-simple tensor category with twist and with only finitely many inequivalent simple objects. The twist in a vertex tensor category is given by the action of $e^{2\pi i L_0}$ with $L_0$ the zero-mode of the Virasoro field. The statement to prove is thus rigidity.
This Theorem is Corollary \ref{cor:ribbon} of the file and it proves Conjecture 1.1 of \cite{CHY} for simply-laced $\g$. The fusion rules are also found as a Corollary of proof, see Corollary \ref{cor:ordinaryfusion}. 
We remark that Corollary 4.2.3 of \cite{FM} is a rigidity statement for $\cO_\ell(\g)$, but it is not explained why the map in their equation (11) is invertible. This  invertibility statement is highly non-trivial as already in the rational case the proof of invertibility of this map was very involved and Yi-Zhi Huang had to first prove and then use Verlinde's formula for it \cite{H1}.

The proof requires to combine insights of three different directions. First of all, we of course need the existence of vertex tensor category structure proven in \cite{CHY}. Secondly, we need a relation of $\AL_\ell(\g)$ to a better understood family of \voas{} and this relation is given by the coset realization of principal W-algebras of simply-laced Lie algebras recently proven in joint work with Arakawa and Linshaw \cite{ACL}. Thirdly we need the theory of vertex algebra extensions developed in collaboration with Kanade and McRae \cite{CKM}. This theory uses that vertex algebra extensions are in one-to-one correspondence to commutative, associative, haploid algebras in the vertex tensor category \cite{FFRS, KO, HKL}.
Combining all these insights we find a fully faithful braided tensor functor from a subcategory of modules of the rational principal W-algebras (rationality is proven in \cite{Ar4}) onto a category that we denote by $\widetilde\cO_\ell(\g)$, this is Theorem \ref{thm:ribbon}. This category inherits the ribbon structure from the subcategory of modules of the rational principal W-algebra and since $\widetilde\cO_\ell(\g)$ and $\cO_\ell(\g)$ differ by the action of certain simple currents the latter is ribbon as well. 

The main Theorem and its proof have quite a few interesting consequences.
One of them is a proof of an admissible level version of a conjecture of Aganagic, Frenkel and Okounkov \cite{AFO} made in the context of the quantum geometric Langlands program, see Remark \ref{rem:AFO}. Combining various coset statements with vertex tensor categories one can prove further statements related to their conjecture and that is work in progress. A second consequence is that one can now study certain \vosas{} that are extensions of affine \voas{} at admissible levels and rational W-algebras. This is done in the instance of $\AL_\ell\left(\mathfrak{osp}(1|2)\right)$ as an extension of $\AL_\ell(\sltwo)$ times a rational Virasoro \voa{} in \cite{CFK, CKLR} and more complicated examples as e.g. $\AL_\ell\left(\mathfrak{osp}(2|2)\right)$ 
and $\AL_1\left(\mathfrak{d}(2, 1; \alpha)\right)$ are feasible future aims. 
The other consequences will now be explained in detail.

\subsection{An ordinary Verlinde's formula}

In order to explain the next result I will explain some well-known background that is quite useful for this work. I assume familiarity with braided tensor categories and refer to the textbook \cite{EGNO} as reference. Let $\cC$ be a braided tensor category with twists and to simplify exposition we assume that it is strict. 
We will denote the tensor bifunctor by $\boxtimes$ and the tensor unit by $\one$. Our field is $\End(\one)=\CC$. The natural families of twists and braidings are denoted by $\theta_\bullet$ and  $c_{\bullet, \bullet}$

The category is called rigid if for each object $M$ in the category, there is a dual
$M^*$ and morphisms $b_M\in\mathrm{Hom}(\one,M\boxtimes M^*)$ (the co-evaluation) and
$d_M\in\mathrm{Hom}(M^*\boxtimes M,\one)$ (the evaluation) such that
\begin{equation}\nonumber
 (\mathrm{Id}_M\boxtimes d_M)\circ(b_M\boxtimes \Id_M)=\Id_M\,,\ \
(d_M\boxtimes\Id_{M^*})\circ(\Id_{M^*}\boxtimes b_M)=\Id_{M^*}.\label{rigid}
\end{equation}
Rigidity implies that there is a trace. Let $f$ in $\End(M)$, then
\[
\text{tr}(f)= d_M\circ c_{M, M^*}\circ\left(\theta_M\boxtimes \Id_{M^*}\right) \circ\left(f\boxtimes \Id_{M^*}\right)\circ b_M\ \in\End(\one) =\CC.
\]
The partial trace is defined similarly. 
 Let $f$ in $\End(M\boxtimes N)$, then
\[
\text{ptr}_L(f)= d_M\circ c_{M, M^*}\circ\left(\theta_M\boxtimes \Id_{M^*}\right) \circ\left(f\boxtimes \Id_{M^*}\right)\circ b_M\ \in\  \End(N).
\]
For us the most important players are the monodromy
\begin{equation}\nonumber
M_{M, N} :=c_{M,N}\circ c_{M,N}  \in  \End(M\boxtimes N)
\end{equation}
and its partial trace and trace the open and closed Hopf links
\begin{equation}\nonumber
\Phi_{M, N} = \mathrm{ptr}_L(M_{M, N})\in \End(N), \qquad
S^\hopflink_{M,N}=
\mathrm{tr}(M_{M, N})\in \CC.
\end{equation}
The use of these lie in the fact \cite{T, BK} that
for any object $N$ of $\cC$, the map 
\[
\Phi_{\,\cdot\,, N}: \mathrm{Obj}(\mathcal C) \rightarrow \mathrm{End}\left(N\right), \qquad M\mapsto \Phi_{M, N}
\]
is a representation of the tensor ring. 
Let us assume that our category $\cC$ is semi-simple so that 
\begin{equation}\nonumber
X \boxtimes Y \cong \bigoplus_{Z \in \text{Sim}(\cC)} N_{X, Y}^{\ \ \ \, Z} \ Z
\end{equation}
where $\text{Sim}(\cC)$  denotes the set of inequivalent simple objects of $\cC$. The collection of numbers $N_{X, Y}^{\ \ \ \, Z}$ for $X, Y, Z$ in $\text{Sim}(\cC)$ are called the fusion rules of $\cC$. We have for any $W$ 
\begin{equation}\nonumber
\Phi_{X, W} \circ \Phi_{Y, W} = \sum_{Z \in \text{Sim}(\cC)} N_{X, Y}^{\ \ \ \, Z} \ \Phi_{Z, W}.
\end{equation}
Let $W$ be simple so that
\begin{equation}\nonumber
\Phi_{X, W} = \frac{S^\hopflink_{X, W}}{S^\hopflink_{\one, W}} \Id_W.
\end{equation}
This identity follows from taking the trace of both sides. 
It follows that 
\begin{equation}\nonumber
\frac{S^\hopflink_{X, W}}{S^\hopflink_{\one, W}}\frac{S^\hopflink_{Y, W}}{S^\hopflink_{\one, W}} \Id_W =   \sum_{Z \in \text{Sim}(\cC)} N_{X, Y}^{\ \ \ \, Z}        \frac{S^\hopflink_{Z, W}}{S^\hopflink_{\one, W}} \Id_W \qquad \text{for all} \ W \ \in \ \text{Sim}(\cC).
\end{equation}
Theorem \ref{thm:hopklinks} says
\begin{introthm}
 Let $\mathfrak g$ be simply laced and let $\ell = -h^\vee +\frac{u}{v}$ be an admissible number for $\g$. Then
open Hopf links $\Phi_{\AL_\ell(\lambda), \AL_\ell(\mu)}$ coincide with character $S^\chi$ in $\cO_\ell(\g)$, i.e.
\[
\Phi_{\AL_\ell(\lambda), \AL_\ell(\mu)} = \frac{S^\hopflink_{\AL_\ell(\lambda), \AL_\ell(\mu)}}{S^\hopflink_{\AL_\ell(0), \AL_\ell(\mu)}} \textup{Id}_{\AL_\ell(\mu)}=\frac{S^\chi_{\AL_\ell(\lambda), \AL_\ell(\mu)}}{S^\chi_{\AL_\ell(0), \AL_\ell(\mu)}}\textup{Id}_{\AL_\ell(\mu)}.
\]
for all simple modules $\AL_\ell(\lambda), \AL_\ell(\mu)$ in $\cO_\ell(\g)$.
\end{introthm}
The character $S^\chi$-matrix is introduced in the main text. 
I call this Theorem a Verlinde formula for ordinary modules. A speculation is that such a formula might hold for ordinary modules of a big class of quasi-lisse \voas{} as they have certain modularity properties \cite{AK}. We remark that also for $C_2$-cofinite \voas{} the open Hopf links seem to be related to the characters $S^\chi$ but there one needs to introduce modified traces \cite{CGan, GRu}.
 We even expect similar behaviour if we go to larger representation categories of $\AL_\ell(\g)$ then just ordinary modules. So far, David Ridout and I, we conjectured a Verlinde formula for relaxed-highest weight modules of $\AL_\ell(\sltwo)$ at admissible level \cite{CR1, CR2} and indeed the conjecture restricted to ordinary modules is true \cite[Cor.\,7.7]{CHY}.

As mentioned before, Yi-Zhi Huang first proved Verlinde's formula for rational \voas{} and then used this result to deduce rigidity \cite{H1, H5}. It might be possible to repeat this analysis for ordinary modules, i.e. first relate a modular $S$-matrix to Hopf links and then use the result to deduce rigidity. 

If the Hopf link matrix is invertible then we can immediately express the fusion rules in terms of the Hopf links, that is
\begin{equation}\nonumber
 N_{X, Y}^{\ \ \ \, Z}   = \sum_{W \in \text{Sim}(\cC)} \frac{S^\hopflink_{X, W}S^\hopflink_{Y, W}\left(S^\hopflink\right)^{-1}_{Z, W}}{S^\hopflink_{\one, W}}.
\end{equation}
This holds if and only if $\cC$ is a modular tensor category. An equivalent formulation of having a modular tensor category is the following. Let 
\begin{equation}\label{eq:W}
W := \bigoplus_{X \in \text{Sim}(\cC)} X
\end{equation}
then $\cC$ is a modular tensor category if and only if the map from the Grothendieck ring $K(\cC)$ to $\End(W)$,
\begin{equation}\label{eq:ringiso}
\Phi_{\, \cdot\, , W} : K(\cC) \rightarrow \End(W), \qquad X \mapsto \Phi_{X, W} 
\end{equation}
is a ring isomorphism. This statement is used to prove 
\begin{introthm}
Let $\mathfrak g$ be simply laced and let $\ell = -h^\vee +\frac{u}{v}$ be an admissible number for $\g$ with $u, v$ positive coprime integers. Let $N$ be the level of the weight lattice $P$ of $\g$ and let $(N, v)=1$. Then $\cO_\ell(\g)$ is a modular tensor category.
\end{introthm}

\subsection{Fusion rules of W-algebras}

Fusion rules of W-algebras are needed for the proof of the first main Theorem. They have been proven in the simply-laced case if a certain coprime property holds for the denominator of the level \cite{FKW, AE}. The proof was based on Verlinde's formula and I take the slightly different point of view that the map \eqref{eq:ringiso} is a ring isomorphism from the Grothendieck ring of the regular W-algebra to the endomorphism ring of the direct sum of all inequivalent simples \eqref{eq:W}.

 It turns out that then fusion rules can be determined in full generality in the sense of the following Theorem. For this let $P_+^m$ ($P_+^{\vee;m}$) be the set of integrable highest-(co)weight modules of $\g$ at level $m\in \mathbb Z_{>0}$, let $h$ ($h^\vee$) be the (dual) Coxeter number of $\g$. Then simple modules of the simple and rational principal W-algebra $W_k(\g)$ are labelled by pairs $(\lambda, \lambda')$ with $\lambda\in P_+^{u-h^\vee}$ and $\lambda'\in P_+^{\vee; v-h}$ where $k=-h^\vee+\frac{u}{v}$ is admissible for $\g$. They are denoted by $\WL_k(\lambda, \lambda')$. This notation is taken from \cite{AE}.
\begin{introthm}
Let $\g$ be a simple Lie algebra
let $k=-h^\vee+\frac{v}{u}$ be an admissible level for $\g$ with $u, v$ positive coprime integers. Let $\ell \in \mathbb Z_{>h^\vee}$ and let $N_{\lambda, \nu}^{\g_\ell, \  \phi}$ the fusion coefficients of  $\AL_{\ell-h^\vee}(\g)$, i.e. they satisfy
\begin{equation}\label{eq:wzwfusion}
\AL_{\ell-h^\vee}(\lambda) \boxtimes \AL_{\ell-h^\vee}(\nu) \cong \bigoplus_{\phi \in P_+^{\ell-h^\vee}} N_{\lambda, \nu}^{\g_\ell  \  \phi}\   \AL_{\ell-h^\vee}(\phi).
\end{equation}
Let $\lambda, \nu \in P_+^{u-h^\vee}$ and $\lambda', \nu' \in P_+^{\vee; v-h}$,
then we have 
\begin{equation}
\begin{split}
\WL_k(\lambda, 0) \boxtimes \WL_k(0, \lambda') &\cong \WL_k(\lambda, \lambda') \\
\WL_{k}(\lambda, 0) \boxtimes \WL_{k}(\nu, 0) &\cong \bigoplus_{\phi \in P_+^{u-h^\vee}} N_{\lambda, \nu}^{\g_u \  \phi} \  \WL_{k}(\phi, 0)
%\WL_{k}(0, \lambda') \boxtimes \WL_{k}(0, \nu') &\cong \bigoplus_{\phi' \in P_+^{\vee; q-h}} N_{\lambda', \nu'}^{{}^L\g_q, \, \  \phi'}   \WL_{k}(0, \phi').
\end{split}
\end{equation}
Moreover $\WL_k(0, \lambda')$ centralizes $\WL_k(\lambda, 0)$ for all $\lambda\in Q$.

Recall that Feigin-Frenkel duality states that $W_k(\g) \cong W_{{}^Lk}({}^L\g)$ for $r^\vee(k+h^\vee)({}^Lk+{}^Lh^\vee)=1$ \textup{\cite{FF, FF2}} (see also \textup{\cite{ACL}})  so that we have 
\[
\WL_{k}(0, \lambda') \boxtimes \WL_{k}(0, \nu') \cong \bigoplus_{\phi' \in {}^LP_+^{q-{}^Lh^\vee}} N_{\lambda', \nu'}^{{}^L\g_v \  \phi'}   \WL_{k}(0, \phi'), \quad  \begin{cases}  q=v &\ \ \text{if} \ (v,r^\vee) =1 \\  q=\frac{v}{r^\vee} &\ \ \text{if} \ (v,r^\vee) =r^\vee \end{cases}
\]
where one identifies coroots with roots of the dual Lie algebra via
${}^L\alpha = \frac{\alpha^\vee}{\sqrt{r^\vee}}$
and correspondingly coweights with weights of the dual Lie algebra. By ${}^LP$ we then mean the weight lattice of the dual Lie algebra ${}^L\g$ and by $r^\vee$ the lacity of $\g$.
\end{introthm}
This is Theorem \ref{thm:Wfusion} together with the first point of Remark \ref{rem:Wfusion}. The fusion rules are completely determined by this Theorem due to associativity and commutativity of fusion. 

That $\WL_k(0, \lambda')$ centralizes $\WL_k(\lambda, 0)$ for all $\lambda\in Q$ (the root lattice) means that the monodromy between these two modules is trivial and this is important for the proof of the first main Theorem as well. 

In \cite{CL1} rationality of subregular $W$-algebras of $\mathfrak{sl}_4$ at certain admissible levels was proven. Moreover it was shown that the Heisenberg coset is isomorphic to a regular $W$-algebra of type $A$. The proof relied on the fusion rules of regular W-algebras of type $A$ and so in that work it could only be done if the coprime condition of \cite{FKW, AE} was satisfied. The exact same proof now works in general and so
\begin{introcor}
Let $n$ be a positive integer, such that $(n+4, n+1)=1$. Let $k=-4+\frac{n+4}{3}$ and let $\ell= -n+\frac{n+4}{n+1}$. Then $W_k(\mathfrak{sl}_4, f_{\text{subregular}})$ is rational and $C_2$-cofinite and let $H$ be its Heisenberg vertex subalgebra, then $\text{Com}(H, W_k(\mathfrak{sl}_4, f_{\text{subregular}})) \cong W_\ell(\mathfrak{sl}_n, f_{\text{regular}})$.
\end{introcor}
These types of relations between regular and subregular W-algebras of type $A$ is expected to hold in general and the cases of the subregular W-algebras of $\sltwo$ and $\mathfrak{sl}_3$ are already known as well \cite{ALY, ACL2}. Moreover, Andrew Linshaw has found a different method to prove these statements \cite[Thms. 10.4, 10.5 and 10.6]{L}.

\subsection*{Acknowledgements}
I am very grateful to Terry Gannon for many discussions on related issues and to Jinwei Yang, Yi-Zhi Huang, Andrew Linshaw, Tomoyuki Arakawa, Shashank Kanade and Robert McRae for the collaborations on the works that were needed for this paper. I am supported by NSERC $\#$RES0020460.

\section{Vertex algebra extensions}

\subsection{Commutative Algebras}

The key player for studying \voa{} extensions are commutative algebras in the tensor category, see \cite{KO, EGNO}. We denote by $\cA_{\bullet, \bullet, \bullet}$ the associativity constraint in $\cC$.
\begin{defi}\label{def:alg}\cite[Def. 7.8.1]{EGNO}
Let $\cC$ be a braided tensor category. A {\bf algebra} in $\cC$ is an object $A$ in $\cC$ together with a multiplication map
\[
m: A \boxtimes A \rightarrow A
\]
and a unit
\[
u : \one \rightarrow A
\]
such that the multiplication is associative and compatible with left and right multiplication, i.e. the following three diagrams commute:
\begin{equation}
\begin{split}
\xymatrix{
\left(A \boxtimes A \right)\boxtimes A  \ar[rr]^(0.5){ \cA_{A, A, A} }\ar[d]_{m \boxtimes \Id_A}
&& A \boxtimes \left(A\boxtimes A\right)  \ar[d]^{\Id_A \boxtimes m} &
\\
A\boxtimes A \ar[rd]_{m}
&& A\boxtimes A \ar[ld]^{m} &
\\
& A &&
\\
} \\ \\
\xymatrix{
\one \boxtimes A \ar[r]^{\ell_A}\ar[d]_{u\boxtimes \Id_A} & A \ar[d]^{\Id_A} && A\boxtimes \one \ar[r]^{r_A}\ar[d]_{\Id_A \boxtimes u} & A \ar[d]^{\Id_A}
\\  
A\boxtimes A \ar[r]^m & A && A\boxtimes A \ar[r]^m & A
\\
}
\end{split}
\end{equation}
The algebra $A$ is called {\bf haploid} if the dimension of $\Hom_\cC(\one, A)$ is one.
$A$ is {\bf commutative} if the diagram 
\begin{equation}
\xymatrix{
A\boxtimes A \ar[rd]_{m}\ar[rr]^{c_{A, A}}
&& A\boxtimes A \ar[ld]^{m} &
\\
& A &&
\\
}
\end{equation}
commutes.
\end{defi}
There are two natural module categories associated to commutative algebras. 
\begin{defi}
Let $\cC$ be a tensor category and $A$ an algebra in $\cC$. Then the category $\cCA$ has objects $(X, m_X)$ with $X$ an object of $\cC$ and $m_X \in \Hom_\cC(A \boxtimes_\cC X, X)$ a multiplication morphism that is
\begin{enumerate}
\item Associative, i.e. the diagram commutes:
\begin{equation}
\begin{split}
\xymatrix{
A \boxtimes (A\boxtimes X)    \ar[rr]^{  \cA^{-1}_{A, A, X} } \ar[d]_{\Id_A \boxtimes m_X } 
&&  (A \boxtimes A) \boxtimes W  \ar[d]^{m \boxtimes \Id_X}
\\
A\boxtimes W \ar[rd]_{m_X} && A\boxtimes W \ar[ld]^{m_X} \\
& X &
\\
} 
\end{split}
\end{equation}
\item Unit: The composition
\[
X \xrightarrow{\ell_X} \one \boxtimes X \xrightarrow{u \boxtimes \Id_X} A \boxtimes X \xrightarrow{m_X} X
\]
is the identity on $X$.
\end{enumerate}
Morphisms of $\cCA$ are all $\cC$ morphisms $f \in \Hom_\cC(X, Y)$ such that 
\begin{equation}
\begin{split}
\xymatrix{
A \boxtimes  X    \ar[rr]^{  \Id_A \boxtimes f   } \ar[d]_{ m_X } 
&&  A \boxtimes Y  \ar[d]^{m_Y }
\\
X \ar[rr]^{f} && Y \\
} 
\end{split}
\end{equation}
commutes.
\end{defi}
The category $\cCA$ is tensor but usually not braided. 
\begin{defi}
Let $\cC$ be a braided tensor category and $A$ a commutative algebra in $\cC$. Then the category $\cCAloc\subset\cCA$ of local modules is the full tensor subcategory whose objects are local with respect to $A$, i.e.
\begin{equation}
\begin{split}
\xymatrix{
A \boxtimes  X    \ar[rd]_{  m_X }  \ar[rr]^{M_{A, X}} && A\boxtimes X \ar[ld]^{m_X} \\
& X&  \\
} 
\end{split}
\end{equation}
commutes. 
\end{defi}
This is the category of interest as it inherits the structure of a braided tensor category of $\cC$.

\section{Vertex Algebra Extensions}

The very important result of Huang, Kirillov and Lepowsky makes contact of \voas{} with a corresponding categorical notion:
\begin{thm}\label{thm:HKL}\textup{\cite[Thms 3.2 and 3.4]{HKL}}
Let $\mathcal C$ be a vertex tensor category of a \voa{} $V$. Then 
a vertex operator algebra extension $V \subset A$ in the category $\mathcal C$ is equivalent to a commutative associative algebra in the braided tensor category $\mathcal C$ with trivial twist and injective unit. Moreover, the category of local $\mathcal C$-algebra modules $\mathcal C_A^{\text{loc}}$ is isomorphic to the category of modules in $\mathcal C$ for the extended vertex operator algebra $A$.
\end{thm}

The first main result of \cite{CKM} is that $\mathcal C_A^{\text{loc}}$ and the vertex tensor category of $A$-modules that lie in $\mathcal C$ are equivalent as braided tensor categories:
\begin{thm}\label{thm:induction}\textup{\cite[Thm. 3.65]{CKM}}
Let $V$ be a \voa{} and let $\cC$ be a full vertex tensor category of $V$-modules. Let $A\in \cC$ be a \voa{} extension of $V$.  Then the isomorphisms of categories of Theorems \ref{thm:HKL}  and the category of vertex algebraic $A$-modules is an equivalence of braided tensor categories. 
\end{thm}
The second main result of \cite{CKM} is that the induction functor is a vertex tensor functor. This is a functor that maps objects of $\mathcal C$ to algebra objects:
\[
\mathcal F: \mathcal C \rightarrow \mathcal C_A, \qquad X\mapsto (A \boxtimes_{\mathcal C} X, m  \boxtimes_{\mathcal C} \Id_X).
\]
\begin{thm}\textup{\cite[Thm. 3.68]{CKM}}
Induction is a vertex tensor functor from the full subcategory of $V$-modules to $\cCA$ with respect to the vertex tensor category structure of these categories constructed by Huang, Lepowsky and Zhang \cite{HLZ8}
\end{thm}
In general the induced object is not local and the criterion that guarantees locality is Proposition 2.65 of \cite{CKM}, saying that $\mathcal F(Y) \in \mathcal C_A^{\text{loc}}$ if and only if $M_{A, Y} = \Id_{A \boxtimes_{\mathcal C} Y}$. One says that $Y$ centralizes $A$. 
\begin{rem}\label{rem:cent}
This centralizing property is sometimes very easy to verify in the instance that $\cC$ is a modular tensor category. Let $\mathcal I$ be a set of inequivalent simple objects such that
\[
A =\bigoplus_{X \in \mathcal I} X
\]
is a commutative algebra object in $\cC$.
Let $Y$ be a simple object of $\cC$ such that $X \boxtimes_{\cC}Y$ is simple for all $X \in \mathcal I$. Then 
\[
M_{X, Y} = a \ \Id_{X \boxtimes_{\cC}Y}
\]
for some number $a$ that is easily computed by taking the trace of both sides, so that
\[
a = \frac{S^\hopflink_{X, Y}}{\dim X \dim Y} 
\]
and thus in this instance $Y$ centralizes $A$ if and only if 
\[
S^\hopflink_{X, Y} = \dim X \dim Y \qquad \text{for all} \ X \ \in  \ \mathcal I.
\]
\end{rem}
The following is a generalization of \cite[Thm. 7.4]{CHY}. The argument of proof is the same.
\begin{thm}\label{thm:liftequi}
Let 
$$
V \cong \bigoplus_i W_i \otimes Z_i
$$ be a \voa{} extending the \voa{} $W \otimes Z$ with $W_i \otimes Z_i$ simple $W \otimes Z$ modules and $W_i \neq  W_j, Z_i \neq  Z_j$ for $i\neq j$.
 Let $\cC_W \boxtimes \cC_Z$ be a vertex tensor category of $W\otimes Z$ modules such that $V$ is an object of $\cC_W \boxtimes \cC_Z$. Let $\cD\subset \cC_W$ be a full rigid braided tensor subcategory of $\cC$, such that every simple object of $\cD$ lifts to a simple local $V$-module under the induction functor $\mathcal F$, i.e.
$\mathcal F(X \boxtimes Z)$ is local and simple for all simple objects $X$ in $\cD$. 
Then $\mathcal F(\cD)$ is rigid braided tensor as well and equivalent to $\cD$ as a braided tensor category.
\end{thm}
\begin{proof}
The induction functor is a braided tensor functor by Theorem 2.67 of \cite{CKM}. It is fully faithful by Proposition 8.1 of \cite{CG}. It thus gives a braided equivalence of $\cD$ and $\mathcal F(\cD)$. Moreover the induction functor preserves duals and hence rigidity by Lemma 1.16 of \cite{KO}, see also Proposition 2.77 of \cite{CKM}.
\end{proof}

\section{Ordinary modules for affine VOAs at admissible level}

I use notation and also many statements of the work of Arakawa and van Ekeren \cite{AE}.
Let $\g$ be a simple Lie algebra and let $k$ be a complex number then we denote the simple affine \voa{} of $\g$ at level $k$ by $\AL_k(\g)$. Let $k=-h^\vee+\frac{u}{v}$ for $u, v$ co-prime positive integers, such that $k$ is a admissible number for $\g$. Let $\AL_k(\lambda)$ be the simple quotient of the highest-weight representation of highest-weight $\lambda$ at level $k$. 
For a positive integer $m$, we denote by $P_+^{m}$ be the set of weights such that $\{ \AL_m(\lambda) | \lambda \in P_+^m\}$ gives the complete set of inequivalent simple modules of $\AL_m(\g)$. Then the category of ordinary modules $\cO_k(\g)$ is a semi-simple category  \cite{Ar3} whose simple objects are $\{ \AL_k(\lambda) | \lambda \in P_+^{u-h^\vee}\}$. It is a vertex tensor category:
\begin{thm}  \textup{\cite[Thm. 6.6]{CHY}}
Let $\g$ be a simple Lie algebra and $k$ a admissible number then the category of ordinary modules $\cO_k(\g)$ has a vertex tensor category structure. 
\end{thm}
Ordinary modules are a subcategory of principal admissible modules at level $k$. These have the property that their trace functions (as Jacobi forms) converge in certain regions and can be meromorphically continued to meromorphic Jacobi forms \cite{KW1}. The latter then turn out to close under the modular group, i.e. let $\ch[M](u, \tau)$ be the meromorphic continuation of the principal admissible module $M$ at level $k$, then especially
\[
\ch[M]\left(\frac{u}{\tau}, -\frac{1}{\tau}\right) = \gamma(u) \sum S^\chi_{M, N} \ch[N](u, \tau)
\] 
where the sum is over all inequivalent simple principal admissible level $k$ modules and $\gamma$ is the usual automorphy factor for Jacobi forms. Then for  $\lambda, \mu \in P_+^{u-h^\vee}$ by \cite{KW1, AE}
\begin{equation}\label{eq:Schaffine}
\frac{S^\chi_{\AL_k(\lambda), \AL_k(\mu)}}{S^\chi_{\AL_k(0), \AL_k(\mu)}} = \frac{\chi_\mu(e(-v/u);\lambda)}{\chi_\mu(e(-v/u); 0)}
\end{equation}
with
\begin{equation}\label{eq:chi}
\chi_\mu(e(-v/u);\lambda):= \sum_{w \in W} \epsilon(w) e^{-2\pi i\frac{v}{u}\left(\lambda+\rho, w(\mu+\rho) \right)}, \qquad \qquad 
\end{equation}
and here $W$ denotes the Weyl group of $\g$ and $\rho$ the Weyl vector. 

\section{W-algebras as coset vertex algebras}

$W$-algebras are \voas{} obtained from affine \voas{} via quantum Hamiltonian reduction. One associates to a Lie algebra $\g$ and a nilpotent element $f$ and a complex number $k$ the $W$-algebra of $\g$ corresponding to $f$ at level $k$. We are interested in regular nilpotent elements $f$ and we denote the resulting $W$-algebra by $W^k(\g)$ and its simple quotient by $W_k(\g)$. This \voa{} is often called the regular or principal $W$-algebra of $\g$ at level $k$.
For literature, we refer to \cite{FF,  Ar1, Ar2}.

Tomoyuki Arakawa has proven \cite{Ar4, Ar5} that $W_k(\g)$ is a rational and $C_2$-cofinite \voa{} if the level satisfies the following conditions
\begin{equation}\nonumber
\begin{split}
k+h^\vee = \frac{u}{v}\, \in \, \mathbb Q_{>0}, \ \ (u, v)=1,  \ \ u, v >0  \ \ \text{and} \ \ \begin{cases} u\geq h^\vee, \ \ v \geq h & \ \ (v, r^\vee) =1 \\ 
 u\geq h, \ \ v \geq r^\vee h^\vee & \ \ (v, r^\vee) =r^\vee \end{cases}
\end{split}
\end{equation}
Here $h$ is the Coxeter number, $h^\vee$ the dual Coxeter number and $r^\vee$ is the lacety of $\g$.
We take the notation of \cite{AE} for W-algebra modules, i.e. simple modules are denoted by $\WL_k(\lambda, \mu)$ with $k+h^\vee=\frac{u}{v}$ and $(\lambda, \mu) \in P_+^{u-h^\vee} \times P_+^{\vee, v-h}$. Moreover $\WL_k(\lambda, \mu)\cong \WL_k(\lambda', \mu')$ if and only if $(\lambda', \mu') = (w(\lambda), w(\mu))$ for some $w\in \widetilde W_+$. Here $ \widetilde W_+$ is the subgroup of automorphisms that preserve the coroot basis. If $\g$ is simply-laced, then it is isomorphic to the discriminant of the root lattice $Q$, i.e. $P/Q$, with $P$ the weight lattice. 

It has been a long standing open conjecture that regular W-algebras of simply-laced Lie algebras have a coset realization inside $\AL_k(\g)\otimes \AL_1(\g)$. In joint work with Tomoyuki Arakawa and Andrew Linshaw this problem has been solved in full generality \cite{ACL}. We need the version at admissible level.
This is the main Theorem 3 part (b) of \cite{ACL}. 
Denote by $W_k(\g)$ the simple regular $W$-algebra of $\g$ at level $k$. 

\begin{thm}\label{thm:GKO}\textup{\cite[Main Theorem 3 (b)]{ACL}}
Let $\ell=-h^\vee +\frac{u}{v}$ be admissible with $(u, v)=1$ and let $\g$ be simply-laced then 
\[
\AL_\ell(\mu) \otimes \AL_1(\nu) \cong  \bigoplus_{\substack{\lambda \in P_+^{u+v-h^\vee}\\ \lambda = \mu+\nu \mod Q}} \AL_{\ell+1}(\lambda) \otimes \WL_{k}(\lambda, \mu)
\]
for $\mu \in  P_+^{u-h^\vee}$ and $k=-h^\vee +\frac{u+v}{u}$. 
\end{thm}
Recall that by the main Theorem of \cite{CHY} we have vertex tensor category structure on $\cO_\ell(\g)$ where the simple objects of $\cO_\ell(\g)$ are the $\AL_\ell(\mu)$ for $\mu$ in  $P_+^{u-h^\vee}$. We also have vertex tensor category on $W_k(g)=\WL_k(0, 0)$-modules as this is a strongly rational \voa{} when the level $k$ is admissible \cite{Ar4}. We will use this fact soon, but first we need to explicitly determine fusion rules of the $W$-algebra.

\section{Fusion rules of $W$-algebras}

The aim of this section is to prove 
\begin{thm}\label{thm:Wfusion}
Let $\g$ be a simple Lie algebra
let $k=-h^\vee+\frac{v}{u}$ be an admissible level for $\g$ with $u, v$ positive integers that satisfy $(u, v)=1$. Let $\ell \in \mathbb Z_{>h^\vee}$ and let $N_{\lambda, \nu}^{\g_\ell, \  \phi}$ the fusion coefficients of  $\AL_{\ell-h^\vee}(\g)$, i.e. they satisfy
\begin{equation}\label{eq:wzwfusion}
\AL_{\ell-h^\vee}(\lambda) \boxtimes \AL_{\ell-h^\vee}(\nu) \cong \bigoplus_{\phi \in P_+^{\ell-h^\vee}} N_{\lambda, \nu}^{\g_\ell  \  \phi}\   \AL_{\ell-h^\vee}(\phi).
\end{equation}
Let $\lambda, \nu \in P_+^{u-h^\vee}$ and $\lambda', \nu' \in P_+^{\vee; v-h}$,
then we have 
\begin{equation}\label{eq:fusionW}
\begin{split}
\WL_k(\lambda, 0) \boxtimes \WL_k(0, \lambda') &\cong \WL_k(\lambda, \lambda') \\
\WL_{k}(\lambda, 0) \boxtimes \WL_{k}(\nu, 0) &\cong \bigoplus_{\phi \in P_+^{u-h^\vee}} N_{\lambda, \nu}^{\g_u \  \phi} \  \WL_{k}(\phi, 0)
%\WL_{k}(0, \lambda') \boxtimes \WL_{k}(0, \nu') &\cong \bigoplus_{\phi' \in P_+^{\vee; q-h}} N_{\lambda', \nu'}^{{}^L\g_q, \, \  \phi'}   \WL_{k}(0, \phi').
\end{split}
\end{equation}
Moreover $\WL_k(0, \lambda')$ centralizes $\WL_k(\lambda, 0)$ for all $\lambda\in Q$.
\end{thm}
\begin{rem}\label{rem:Wfusion} Three remarks are in order:
\begin{enumerate}
\item
Recall that Feigin-Frenkel duality states that $W_k(\g) \cong W_{{}^Lk}({}^L\g)$ for $r^\vee(k+h^\vee)({}^Lk+{}^Lh^\vee)=1$ \cite{FF, FF2, ACL}  so that we have 
\[
\WL_{k}(0, \lambda') \boxtimes \WL_{k}(0, \nu') \cong \bigoplus_{\phi' \in {}^LP_+^{q-{}^Lh^\vee}} N_{\lambda', \nu'}^{{}^L\g_q \  \phi'}   \WL_{k}(0, \phi'), \quad  \begin{cases}  q=v &\ \ \text{if} \ (v,r^\vee) =1 \\  q=\frac{v}{r^\vee} &\ \ \text{if} \ (v,r^\vee) =r^\vee \end{cases}
\]
where one identifies coroots with roots of the dual Lie algebra via
${}^L\alpha = \frac{\alpha^\vee}{\sqrt{r^\vee}}$
and correspondingly coweights with weights of the dual Lie algebra. By ${}^LP$ we then mean the weight lattice of the dual Lie algebra ${}^L\g$. 
\item Fusion rules have been known for $\g$ simply-laced, see \cite[Thm. 8.5]{AE} and \cite[Thm. 4.3]{FKW} but with the assumption that $q$ is coprime to the order of the discriminant $P/Q$.
Together with commutativity and associativity the fusion rules stated by us completely determine all fusion rules:
\begin{equation}\nonumber
\begin{split}
 \WL_k(\lambda, \lambda')  \boxtimes  \WL_{k}(\nu, \nu') &\cong \WL_k(\lambda, 0) \boxtimes \WL_k(0, \lambda') \boxtimes \WL_k(\nu, 0) \boxtimes \WL_k(0, \nu') \\
 &\cong  \WL_{k}(\lambda, 0) \boxtimes \WL_{k}(\nu, 0) \boxtimes \WL_{k}(0, \lambda') \boxtimes \WL_{k}(0, \nu')  \\
 &\cong \bigoplus_{\phi \in P_+^{u-h^\vee}} N_{\lambda, \nu}^{\g_u \  \phi}   \WL_{k}(\phi, 0)  \boxtimes \bigoplus_{\phi' \in {}^LP_+^{q-{}^Lh^\vee}} N_{\lambda', \nu'}^{{}^L\g_q \  \phi'}   \WL_{k}(0, \phi') \\
 &\cong \bigoplus_{\substack {\phi \in P_+^{u-h^\vee} \\ \phi' \in {}^LP_+^{q-{}^Lh^\vee}}}  N_{\lambda, \nu}^{\g_u \  \phi}  N_{\lambda', \nu'}^{{}^L\g_q \  \phi'}   \WL_{k}(\phi, \phi').
\end{split}
\end{equation}
\item The fusion rules of \eqref{eq:fusionW} tell us that the map $\AL_{\ell-h^\vee}(\lambda) \mapsto \WL_k(\lambda, 0)$ gives a ring homomorphism
from the tensor ring of the affine \voa{} to the one of the W-algebra. It is not injective if $v-h=0$.
\end{enumerate}
\end{rem}
Recall from the introduction that a modular tensor category has the property that the map 
\begin{equation}\label{eq:hopfgrothendieck1}\nonumber
K(\cC) \rightarrow \text{End}\Big(\bigoplus_{W \in  \text{Sim}(\cC)} W\Big), \qquad X \mapsto \sum_{W \in  \text{Sim}(\cC)} \Phi_{X, W} = \sum_{W \in  \text{Sim}(\cC)}  \frac{S^\hopflink_{X, W}}{S^\hopflink_{0, W}}\Id_{W}
\end{equation}
is a ring isomorphism. Heret the set of inequivalent simple objects of $\cC$ is denoted by $\text{Sim}(\cC)$.  We can thus prove fusion rules by computing all open Hopf links.
The rest of the section is the proof of the Theorem:
\begin{proof}
We continue to take notation and set-up of \cite{AE}. 
Let $\g$ be a simple Lie algebra and
let $k=-h^\vee+\frac{v}{u}$ be an admissible level for $\g$ and let $\lambda, \mu \in P_+^{u-h^\vee}$ and $\lambda', \mu' \in P_+^{\vee; v-h}$. Let $I_+^{(p, q)} = \left( P_+^{u-h^\vee} \times P_+^{\vee; v-h}\right)/ \widetilde W_+$ and 
recall the definition of $\chi_\mu(e(-v/u);\lambda)$ in \eqref{eq:chi}.
Let me take the short hand notation for  this proof
\[
\chi_\mu(e(-v/u);\lambda):= \chi_\mu(\lambda), \quad 
\chi_{\mu'}(e(-u/v);\lambda'):=\chi'_{\mu'}(\lambda')= \sum_{w \in W} \epsilon(w) e^{-2\pi i\frac{u}{v}\left(\lambda'+\rho, w(\mu'+\rho) \right)}
\]
so that the modular $S$-matrix is given by \cite[Cor. 8.4]{AE}
\begin{equation}
\begin{split}
\ch[\WL_k(\lambda, \lambda')]\left(-\frac{1}{\tau}, \tau^{-L_{[0]}}u \right) &= \sum_{\mu, \mu'\in I_+^{p, q}}  S^\chi_{(\lambda, \lambda'),(\mu, \mu')}\ch[\WL_k(\mu, \mu')]\left(\tau, u \right)  \\
S^\chi_{(\lambda, \lambda'),(\mu, \mu')} &= \frac{e^{2\pi i \left((\lambda+\rho, \mu'+\rho)+(\lambda'+\rho, \mu+\rho)\right)}}{(uv)^{}\ell/2 |J|^{1/2}} \chi_\mu(\lambda)\chi'_{\mu'}(\lambda').
\end{split}
\end{equation}
$\ell$ is the rank of $\g$ and $J$ is the groupoid of simple currents. 
We note that \cite{AE} introduced a second variable $u$ in order to distinguish characters that otherwise would be linearly dependent. 
This then confirmed an earlier result \cite{FKW}.
It follows that $\chi_\mu(0)$ and $\chi'_{\mu'}(0)$ are non-zero as the modular $S^\chi$-matrix and the Hopf link $S^\hopflink$ in a rational and $C_2$-cofinite \voa{} are proportional to each other (and non-zero) and the dimension of a simple object in a semi-simple rigid tensor category is always non-zero. We thus have
\[
\frac{S^\chi_{(\lambda, \lambda'),(\mu, \mu')}}{S^\chi_{(0, 0),(\mu, \mu')}} = \frac{e^{2\pi i \left((\lambda', \mu+\rho)+(\lambda, \mu'+\rho)\right)} \chi_\mu(\lambda)\chi'_{\mu'}(\lambda')}{\chi_\mu(0)\chi'_{\mu'}(0)}.
\]
Hence for any $\mu, \mu'$ we have that 
\begin{equation}\nonumber
\begin{split}
\frac{S^\chi_{(\lambda, \lambda'),(\mu, \mu')}}{S^\chi_{(0, 0),(\mu, \mu')}} &= \frac{e^{2\pi i \left((\lambda', \mu+\rho)+(\lambda, \mu'+\rho)\right)} \chi_\mu(\lambda)\chi'_{\mu'}(\lambda')}{\chi_\mu(0)\chi'_{\mu'}(0)}\\
&=  \frac{e^{2\pi i \left((\lambda, \mu'+\rho)+(0, \mu+\rho)\right)} \chi_\mu(\lambda)\chi'_{\mu'}(0)}{\chi_\mu(0)\chi'_{\mu'}(0)}  \frac{e^{2\pi i \left((0, \mu'+\rho)+(\lambda', \mu+\rho)\right)} \chi_\mu(0)\chi'_{\mu'}(\lambda')}{\chi_\mu(0)\chi'_{\mu'}(0)}\\
&= \frac{S^\chi_{(\lambda, 0),(\mu, \mu')}}{S^\chi_{(0, 0),(\mu, \mu')}}\frac{S^\chi_{(0, \lambda'),(\mu, \mu')}}{S^\chi_{(0, 0),(\mu, \mu')}}.
\end{split}
\end{equation}
Since the $W$-algebra is rational \cite{Ar4} so that its module category is modular \cite{H1} the map from the Grothendieck ring to the endomorphism ring of the direct sum of all simples via open Hopf links is an isomorphism and so it follows that
\[
\WL_k(\lambda, 0) \boxtimes \WL_k(0, \lambda') \cong \WL_k(\lambda, \lambda').
\]
We now check the centralizing property.
Using the symmetry that $\chi_\lambda(\mu)=\chi_\mu(\lambda)$ and that for any two objects $X, Y$ of a strongly rational \voa{} one has 
\[
\frac{S^\hopflink_{X,Y}}{S^\hopflink_{\one, Y}} = \frac{S^\chi_{X,Y}}{S^\chi_{\one, Y}}
\]
 we get that
\begin{equation}\nonumber
\begin{split}
\frac{S^\hopflink_{(\lambda, 0),(0, \lambda')}}{S^\hopflink_{(0, 0),(0, \lambda')}S^\hopflink_{(0, 0),(\lambda, 0)}}&= \frac{S^\hopflink_{(\lambda, 0),(0, \lambda')}}{S^\hopflink_{(0, 0),(\lambda, \lambda')}}  
=  \frac{S^\hopflink_{(\lambda, 0),(0, \lambda')}}{S^\hopflink_{(0, 0),(\lambda, \lambda')}}  \frac{S^\hopflink_{(0, 0),(0, \lambda')}}{S^\hopflink_{(0, 0),(0, \lambda')}}   \frac{S^\hopflink_{(0, 0),(0, 0)}}{S^\hopflink_{(0, 0),(0, 0)}} \\
&=  \frac{S^\chi_{(\lambda, 0),(0, \lambda')}}{S^\chi_{(0, 0),(\lambda, \lambda')}}  \frac{S^\chi_{(0, 0),(0, \lambda')}}{S^\chi_{(0, 0),(0, \lambda')}}   \frac{S^\chi_{(0, 0),(0, 0)}}{S^\chi_{(0, 0),(0, 0)}} =
\frac{S^\chi_{(\lambda, 0),(0, \lambda')}}{S^\chi_{(0, 0),(\lambda, \lambda')}}\\
 &= \frac{e^{2\pi i \left((\lambda'+\rho, \lambda+\rho)+(\rho, \rho)\right)} \chi_0(\lambda)\chi'_{\lambda'}(0)}{e^{2\pi i \left((\lambda+\rho, \rho)+(\lambda'+\rho, \rho)\right)}  \chi_\lambda(0)\chi'_{\lambda'}(0)} =  \frac{e^{2\pi i \left((\lambda'+\rho, \lambda+\rho)+(\rho, \rho)\right)}}{e^{2\pi i \left((\lambda+\rho, \rho)+(\lambda'+\rho, \rho)\right)} }\\
&= e^{2\pi i (\lambda, \lambda')}.
\end{split}
\end{equation}
It follows that $ \WL_k(0, \lambda')$ centralizes $\WL_k(\lambda, 0)$ for all $\lambda\in Q$ by Remark \ref{rem:cent}. 

We will now be more explicit on the fusion ring. For this let $r \in \{1, v\}$ and consider the ring $R_r= \Z(e(-r/(Nu)))$ with $N$ the level of the weight lattice $P$, that is the smallest positive integer $N$, such that $NP$ is an integral lattice. For example for $\mathfrak{sl}_n$ we have $n=N$. 
Let $S_r=<\chi_\mu(e(-r/u);0)>$ be the monoid generated by $\chi_\mu(e(-r/u);0)$ so that 
\begin{equation}\nonumber
\begin{split}
\frac{S^\chi_{(\lambda, 0),(\mu, \mu')}}{S^\chi_{(0, 0),(\mu, \mu')}} &= \frac{e^{2\pi i \left((\lambda, \mu'+\rho)\right)} \chi_\mu(\lambda)}{\chi_\mu(0)} \ \in  \ S_v^{-1}R_v.
\end{split}
\end{equation}
Let $\sigma_v$ be the surjective ring homomorphism generated by $e(-1/(Nu))\mapsto e(-v/(Nu))$, i.e.
\begin{equation}\label{eq:galois}
\sigma_v : S_1^{-1}R_1 \rightarrow S_v^{-1}R_v, \qquad e(-1/(Nu))\mapsto e(-v/(Nu)).
\end{equation}
 So that we see that
\[
\sigma_v\left(\frac{\chi_\mu(e(-1/u);\lambda)}{\chi_\mu(e(-1/u);0)} \right) =\frac{\chi_\mu(e(-v/u);\lambda)}{\chi_\mu(e(-v/u);0)}.
\]
The former are the normalized Hopf links of $\AL_{u-h^\vee}(\g)$ and satisfy
\[
\frac{\chi_\mu(e(-1/u);\lambda)}{\chi_\mu(e(-1/u);0)}  \frac{\chi_\mu(e(-1/u);\nu)}{\chi_\mu(e(-1/u);0)} = \sum_{\phi} N_{\lambda, \nu}^{\g_u \ \phi} \frac{\chi_\mu(e(-1/u);\phi)}{\chi_\mu(e(-1/u);0)} 
\]
with $N_{\lambda, \nu}^{\g_u\ \phi}$ the fusion coefficients of  $\AL_{u-h^\vee}(\g)$:
\[
\AL_{u-h^\vee}(\lambda) \boxtimes \AL_{u-h^\vee}(\nu) \cong \bigoplus_{\phi \in P_+^{u-h^\vee}} N_{\lambda, \nu}^{\g_u  \  \phi}   \AL_{u-h^\vee}(\phi).
\]
Note that $N_{\lambda, \nu}^{\g_u \  \phi}\neq  0$ implies that $\lambda+\nu = \phi \mod Q$ since the fusion product is a homomorphic image of the tensor ring of the compact Lie group. It thus follows that
\begin{equation}\nonumber
\begin{split}
\frac{S^\chi_{(\lambda, 0),(\mu, \mu')}}{S^\chi_{(0, 0),(\mu, \mu')}} \frac{S^\chi_{(\nu, 0),(\mu, \mu')}}{S^\chi_{(0, 0),(\mu, \mu')}}  &= e^{2\pi i \left((\lambda+\nu, \mu'+\rho)\right)} \frac{\chi_\mu(\lambda)}{\chi_\mu(0)} \frac{\chi_\mu(\nu)}{\chi_\mu(0)} \\ 
&=e^{2\pi i \left((\lambda+\nu, \mu'+\rho)\right)} \sigma_v\left(\frac{\chi_\mu(e(-1/u);\lambda)}{\chi_\mu(e(-1/u);0)} \right)\sigma_v\left(\frac{\chi_\mu(e(-1/u);\nu)}{\chi_\mu(e(-1/u);0)} \right) \\
&=e^{2\pi i \left((\lambda+\nu, \mu'+\rho)\right)}  \sum_{\phi} N_{\lambda, \nu}^{\g_u  \ \phi} \sigma_v\left(\frac{\chi_\mu(e(-1/u);\phi)}{\chi_\mu(e(-1/u);0)} \right) \\
&=  \sum_{\phi} N_{\lambda, \nu}^{\g_u \ \phi} e^{2\pi i \left((\phi, \mu'+\rho)\right)} \sigma_q\left(\frac{\chi_\mu(e(-1/u);\phi)}{\chi_\mu(e(-1/u);0)} \right) \\
&=  \sum_{\phi} N_{\lambda, \nu}^{\g_u \ \phi} \frac{S^\chi_{(\phi, 0),(\mu, \mu')}}{S^\chi_{(0, 0),(\mu, \mu')}}, 
\end{split}
\end{equation}
and hence
\[
\WL_{k}(\lambda, 0) \boxtimes \WL_{k}(\nu, 0) \cong \bigoplus_{\phi \in P_+^{u-h^\vee}} N_{\lambda, \nu}^{\g_u\ \phi}   \WL_{k}(\phi, 0).
\]
\end{proof}

\section{Fusion categories for ordinary modules}

Let $\g$ be simply-laced. We can now proceed in analogy to section 7 of \cite{CHY}. Firstly, let $\ell+h^\vee=\frac{u}{v}$ and let $\mathcal C_u$ be the category of right modules of $W_k(\g)$ with $k+h^\vee=\frac{u+v}{u}$, i.e. the simple objects are the $\WL_k(0, \mu)$ with $\mu \in P^{u-h^\vee}_+$ and we have just seen that this is a full vertex tensor subcategory for $W_k(\g)$. Moreover, consider $\cO_\ell(\g) \boxtimes \cO_1(\g)$ and let $\widetilde \cO_\ell(g)$ be its subcategory whose simple objects are the $\AL_\ell(\lambda) \otimes \AL_1(\nu)$ with $(\lambda, \nu) \in  P^{u-h^\vee}_+ \times  P^{1}_+$ such that $\lambda+\nu \in Q$. Then
\begin{thm}\label{thm:ribbon}
The categories $\widetilde \cO_\ell(g)$ and $\cC_u$ are equivalent as braided tensor categories with twist. In particular  $\widetilde \cO_\ell(g)$ is rigid and thus it is a ribbon category.
\end{thm}
\begin{rem}\label{rem:AFO}
Aganagic, Frenkel and Okounkov \cite{AFO} conjecture a braided equivalence between $W$-algebra modules and Weyl modules of the affine \voa{} at generic levels if they satisfy the relation
\[
\beta= \frac{1}{k+h^\vee} +r^\vee N
\]
where $r^\vee$ is the lacety of $\g$ and $\beta$ is the shifted level of the $W$-algebra. The original conjecture (Conjecture 6.3 of \cite{AFO}) is the case $N=0$, but they later say that this should generalize to any integral $N$. 
Our Theorem here is the simply-laced version at admissible level of this at $N=1$ and we see that the statement has to be slightly tweaked as we get an equivalence between 
$\widetilde \cO_\ell(g)$ and $\cC_u$. 
Note that $\beta= \frac{u+v}{u}$ here and $k+h^\vee=\frac{u}{v}$ so this fits. 
\end{rem}
\begin{proof}
Let $V=\AL_{\ell+1}(\g) \otimes W_k(g)$ and let $A=\AL_\ell(\g) \otimes \AL_1(\g)$. Then $V$ has a vertex tensor category structure $\cC_V$ given by the Deligne product of the category of ordinary modules of $\AL_{\ell+1}$ and the modular tensor category of $W_k(\g)$. By Theorem \ref{thm:GKO} the \voa{} $A$ is an object in this category and thus by Theorem \ref{thm:HKL} corresponds to a commutative and associative algebra in $\cC_V$. 
We can view $\cC_u$ as a full tensor subcategory of $\cC_V$ via the embedding
\[
\cC_u \rightarrow \cC_V, \qquad \WL_k(0, \mu)  \mapsto \AL_{\ell+1}(\g) \boxtimes  \WL_k(0, \mu).
\] 
Let $\cF: \cC_V \rightarrow \cCA$ be the induction functor. By Theorem \ref{thm:Wfusion} every object in $\cC_u$ centralizes $A$ and so $\cF$ maps $\cC_u$ to $\cCAloc$. We are thus in the situation of Theorem \ref{thm:liftequi} and the image of $\AL_{\ell+1}(\g) \boxtimes  \WL_k(0, \mu)$ under induction is
\begin{equation}\nonumber
\begin{split}
\cF(\AL_{\ell+1}(\g) \boxtimes  \WL_k(0, \mu)) &\cong_V A\boxtimes_V  \AL_{\ell+1}(\g) \boxtimes  \WL_k(0, \mu) \\
&\cong_V\bigoplus_{\lambda \in P_+^{u+v-h^\vee}\cap Q} \left(\AL_{\ell+1}(\lambda) \otimes \WL_{k}(\lambda, 0)\right)\boxtimes_V  \left(\AL_{\ell+1}(\g) \boxtimes  \WL_k(0, \mu)\right) \\ 
&\cong_V \bigoplus_{\lambda \in P_+^{u+v-h^\vee}\cap Q} \AL_{\ell+1}(\lambda) \otimes \WL_{k}(\lambda, \mu)\\
&\cong_V \AL_\ell(\mu) \otimes \AL_1(\nu) \ \in  \ \widetilde \cO_\ell(\g).
\end{split}
\end{equation}
Here we used our fusion rules as well as  Theorem \ref{thm:GKO} for the identification with $\AL_\ell(\mu) \otimes \AL_1(\nu)\in  \widetilde \cO_\ell(\g)$.
The weight $\nu$ is the unique one of  $\AL_1(\g)$ such that $\mu+\nu\in Q$. 
The claim is now proven due to Theorem \ref{thm:liftequi}.
\end{proof}
This Theorem has a series of nice corollaries. 
\begin{cor}\label{cor:ribbon}
The category $\cO_\ell(\g)$ is ribbon.
\end{cor}
\begin{proof}
The category $\cO_\ell(\g)$ is a subcategory of $\cO_\ell(\g) \boxtimes \cO_1(\g)$ under the identification of $\AL_\ell(\mu)$ with $ \AL_\ell(\mu) \otimes \AL_1(0)$.
The $\AL_1(\nu)$ are all rigid simple currents with fusion rules $\AL_1(\nu) \boxtimes \AL_1(\nu') \cong \AL_1(\nu+\nu')$ and $\nu, \nu'\in P/Q$ and via the embedding $\cO_1(\g) \rightarrow \cO_\ell(\g) \boxtimes \cO_1(\g)$ the objects $\AL_\ell(0) \otimes \AL_1(\mu)$ are rigid simple currents as well. The inverse object of $\AL_\ell(0) \otimes \AL_1(\mu)$ is the dual module denoted by $\AL_\ell(0) \otimes \AL_1(\mu^*)$.
It follows that
\[
\AL_\ell(\mu) \otimes \AL_1(0) \cong \left(\AL_\ell(\mu) \otimes \AL_1(\mu^*)\right) \boxtimes \left(\AL_\ell(0) \otimes \AL_1(\mu)\right)
\]
and hence $\AL_\ell(\mu) \otimes \AL_1(0)$ is rigid as tensor product of two rigid modules is rigid.
\end{proof}
\begin{cor}\label{cor:ordinaryfusion}
The fusion rules in $\cO_\ell(\g)$ are
\[
\AL_{\ell}(\lambda) \boxtimes \AL_{\ell}(\nu) \cong \bigoplus_{\phi \in P_+^{u-h^\vee}} N_{\lambda, \nu}^{\g_u  \  \phi}   \AL_{\ell}(\phi)
\]
with $N_{\lambda, \nu}^{\g_u \  \phi}$ the fusion rules of $\AL_{u-h^\vee}(\g)$, see \eqref{eq:wzwfusion}. 
\end{cor}
\begin{proof}
This is clear, since the induction functor is a tensor functor and since the fusion rules of $\cC_u$ have been determined in Theorem \ref{thm:Wfusion} and since $\AL_1(\nu) \boxtimes \AL_1(\nu') \cong \AL_1(\nu+\nu')$ for $\nu, \nu'\in P/Q$ together with $N_{\lambda, \nu}^{\g_u \  \phi}=0$ unless $\lambda+\nu=\phi\mod Q$. 
\end{proof}

\section{Hopf links $S^\hopflink$ coincide with character $S^\chi$}

We now compare the open Hopf links of admissible affine \voa{} to modular $S^\chi$. For this let $\mathfrak g$ be simply laced and let $\ell = -h^\vee +\frac{u}{v}$ be an admissible number. 
\begin{thm}\label{thm:hopklinks}
Open Hopf links coincide with normalized  character $S^\chi$ in $\cO_\ell(\g)$, i.e.
\[
\frac{S^\hopflink_{\AL_\ell(\lambda), \AL_\ell(\mu)}}{S^\hopflink_{\AL_\ell(0), \AL_\ell(\mu)}}=\frac{S^\chi_{\AL_\ell(\lambda), \AL_\ell(\mu)}}{S^\chi_{\AL_\ell(0), \AL_\ell(\mu)}}= \frac{\chi_\mu(e(-v/u);\lambda)}{\chi_\mu(e(-v/u); 0)}
\]
for all $\lambda, \mu \in P^{u-h^\vee}_+$.
\end{thm}
\begin{proof}
Theorem 2.89 of \cite{CKM} tells us that induction preserves Hopf links, so that the Hopf links of $W_k(\g)$ with $k=-h^\vee +\frac{u+v}{u}$ satisfy
\begin{equation}\label{eq:hopfchi1}
\frac{S^\hopflink_{(0, \lambda),(0, \mu)}}{S^\hopflink_{(0, 0),(0, \mu)}} =  \frac{S^\hopflink_{\AL_1(\lambda^*), \AL_1(\mu^*)}}{S^\hopflink_{\AL_1(0), \AL_1(\mu^*)}}\frac{S^\hopflink_{\AL_\ell(\lambda), \AL_\ell(\mu)}}{S^\hopflink_{\AL_\ell(0), \AL_\ell(\mu)}}
\end{equation}
The Hopf links of interest of $W_k(\mathfrak g)$ are given by Verlinde's formula \cite{H1} via character's $S^\chi$ by 
\begin{equation}\nonumber
\begin{split}
\frac{S^\hopflink_{(0, \lambda),(0, \mu)}}{S^\hopflink_{(0, 0),(0, 0)}} &= \frac{S^\hopflink_{(0, \lambda),(0, \mu)}}{S^\hopflink_{(0, 0),(0, \mu)}} \frac{S^\hopflink_{(0, 0),(0, \mu)}}{S^\hopflink_{(0, 0),(0, 0)}} 
= \frac{S^\chi_{(0, \lambda),(0, \mu)}}{S^\chi_{(0, 0),(0, \mu)}} \frac{S^\chi_{(0, 0),(0, \mu)}}{S^\chi_{(0, 0),(0, 0)}}
=\frac{S^\chi_{(0, \lambda),(0, \mu)}} {S^\chi_{(0, 0),(0, 0)}}\\
&= e^{2\pi i(\lambda+\mu, \rho) }\frac{\chi_\mu(e(-(u+v)/u);\lambda)}{\chi_0(e(-(u+v)/u); 0)}
= e^{-2\pi i ((\mu, \lambda)+\rho^2)}  \frac{\chi_\mu(e(-v/u);\lambda)}{\chi_0(e(-v/u); 0)}
\end{split}
\end{equation}
where for the last equality we used that Weyl reflections change weights by an element in the root lattice and since 
\[
e^{-2\pi i (\mu+\rho, \lambda+\rho)} = e^{-2\pi i (w(\mu+\rho), \lambda+\rho)}
\]
for any $w\in W$ we can take this factor out of the sum. It follows that
\begin{equation}\nonumber
\frac{S^\hopflink_{(0, \lambda),(0, \mu)}}{S^\hopflink_{(0, 0),(0, \mu)}} = \frac{S^\hopflink_{(0, \lambda),(0, \mu)}}{S^\hopflink_{(0, 0),(0, 0)}} \frac{S^\hopflink_{(0, 0),(0, 0)}}{S^\hopflink_{(0, 0),(0, \mu)}} = e^{-2\pi i (\mu, \lambda)}  \frac{\chi_\mu(e(-v/u);\lambda)}{\chi_\mu(e(-v/u); 0)}
\end{equation}
but this coincides with the normalized character $S$-matrix of the affine \voa{} at level $\ell$ times the corresponding lattice \voa{} module, i.e.
\begin{equation}\label{eq:hopfchi2}
\frac{S^\hopflink_{(0, \lambda),(0, \mu)}}{S^\hopflink_{(0, 0),(0, \mu)}} =  \frac{S^\chi_{\AL_1(\lambda^*), \AL_1(\mu^*)}}{S^\chi_{\AL_1(0), \AL_1(\mu^*)}}\frac{S^\chi_{\AL_\ell(\lambda), \AL_\ell(\mu)}}{S^\chi_{\AL_\ell(0), \AL_\ell(\mu)}}.
\end{equation}
By Verlinde's formula \cite{H1}  we have 
\[
 \frac{S^\chi_{\AL_1(\lambda^*), \AL_1(\mu^*)}}{S^\chi_{\AL_1(0), \AL_1(\mu^*)}}= \frac{S^\hopflink_{\AL_1(\lambda^*), \AL_1(\mu^*)}}{S^\hopflink_{\AL_1(0), \AL_1(\mu^*)}}
\]
and since all $\AL_1(\mu)$ are simple currents in a ribbon category the normalized Hopf links must be roots of unity of order dividing the order of the simple current. They are especially all non-zero 
so that the claim follows by comparing \eqref{eq:hopfchi1} and \eqref{eq:hopfchi2}.
\end{proof}
Recall that a ribbon category $\cC$ is a modular tensor category if and only if the map 
\begin{equation}\label{eq:hopfgrothendieck}
K(\cC) \rightarrow \text{End}\Big(\bigoplus_{W \in  \text{Sim}(\cC)} W\Big), \qquad X \mapsto \sum_{W \in  \text{Sim}(\cC)} \Phi_{X, W} = \sum_{W \in  \text{Sim}(\cC)}  \frac{S^\hopflink_{X, W}}{S^\hopflink_{0, W}}\Id_{W}
\end{equation}
is a ring isomorphism. Recall also that the set of inequivalent simple objects of $\cC$ is denoted by $\text{Sim}(\cC)$.  
\begin{cor}\label{cor:modular}
With the same notation as in the previous Theorem. Let $N$ be the level of the weight lattice $P$ of $\g$ and let $(N, v)=1$. Then $\cO_\ell(\g)$ is a modular tensor category.
\end{cor}
\begin{proof}
If $(N, v)=1$ then the homomorphism $\sigma_v$ of \eqref{eq:galois} induces a Galois automorphism on $\Q(e(-1/(Nu)))$. 
We have just proven that
\[
\frac{S^\hopflink_{\AL_\ell(\lambda), \AL_\ell(\mu)}}{S^\hopflink_{\AL_\ell(0), \AL_\ell(\mu)}}= \frac{\chi_\mu(e(-v/u);\lambda)}{\chi_\mu(e(-v/u); 0)}
\]
so that we obviously have 
\[
\sigma_v\left(\frac{S^\hopflink_{\AL_{u-h^\vee}(\lambda), \AL_{u-h^\vee}(\mu)}}{S^\hopflink_{\AL_{u-h^\vee}(0), \AL_{u-h^\vee}(\mu)}}\right)    =\frac{S^\hopflink_{\AL_\ell(\lambda), \AL_\ell(\mu)}}{S^\hopflink_{\AL_\ell(0), \AL_\ell(\mu)}}
\]
for all $\lambda, \nu$ in $P_+^{u-h^\vee}$. The map \eqref{eq:hopfgrothendieck} is a ring isomorphism for the modular tensor category $\cC=\cO_{u-h^\vee}(\g)$ so that this map is also a ring isomorphism for the ribbon category $\cC=\cO_\ell(\g)$.
\end{proof}

\end{document}